\documentclass[11pt]{article}
\usepackage{amsmath,amsthm}
\usepackage{amssymb,mathrsfs}
\usepackage{bm,mathtools}
\usepackage{xcolor}
\usepackage{geometry}
\usepackage[colorlinks=true,linkcolor=blue,citecolor=blue,
urlcolor=blue]{hyperref}

\makeatletter
\def\@seccntDot{.}
\def\@seccntformat#1{\csname the#1\endcsname\@seccntDot\hskip 0.5em}
\renewcommand\section{\@startsection{section}{1}{\z@}%
	{18\p@ \@plus 6\p@ \@minus 3\p@}%
	{9\p@ \@plus 6\p@ \@minus 3\p@}%
	{\large\bfseries\boldmath}}
\renewcommand\subsection{\@startsection{subsection}{2}{\z@}%
	{12\p@ \@plus 6\p@ \@minus 3\p@}%
	{3\p@ \@plus 6\p@ \@minus 3\p@}%
	{\bfseries\boldmath}}
\renewcommand\subsubsection{\@startsection{subsubsection}{3}{\z@}%
	{12\p@ \@plus 6\p@ \@minus 3\p@}%
	{\p@}%
	{\bfseries\boldmath}}
\makeatother

\usepackage{microtype}

\theoremstyle{plain}
\newtheorem{thm}{Theorem}[section]
\newtheorem{lem}[thm]{Lemma}
\newtheorem{problem}[thm]{Problem}

\theoremstyle{definition}
\newtheorem{definition}[thm]{Definition}

\numberwithin{equation}{section}
\allowdisplaybreaks
\newcommand{\ex}{\operatorname{ex}}
\newcommand{\EX}{\operatorname{EX}}
\newcommand{\SPEX}{\operatorname{SPEX}}
\newcommand{\spex}{\operatorname{spex}}

\title{Sharp spectral Moon--Moser-type theorems in the linear range via feasible graph parameters}
\author{
Yang Hu\thanks{Email address: \texttt{yang.hu@mail.nankai.edu.cn}}\\
College of Computer Science, Nankai University, Tianjin, 300350, China
}
\date{}

\begin{document}
\maketitle

\begin{abstract}
	Moon and Moser proved a sharp edge-extremal theorem for Hamilton cycles in balanced bipartite graphs with minimum degree at least $k$. Li and Ning obtained spectral analogues for Hamiltonicity in balanced bipartite graphs of order $2n$ and for traceability in nearly balanced bipartite graphs with part sizes $n$ and $n-1$, under the assumption $n\ge (k+1)^2$. We show that their sharp spectral thresholds remain valid in the linear ranges $n\ge 2k$ and $n\ge 2k+1$, respectively.
	
	More precisely, we determine the extremal values of the adjacency spectral radius and the signless Laplacian spectral radius for non-Hamiltonian balanced bipartite graphs with minimum degree $\delta(G)\ge k$, and for non-traceable nearly balanced bipartite graphs with $\delta(G)\ge k$. In each case, the extremal graph is unique up to isomorphism.
	
	Our proof is based on feasible graph parameters: parameters that increase under edge addition and are nondecreasing under Kelmans operations. This yields Moon--Moser type extremal theorems for a general class of parameters, from which the spectral results follow.
	\par\vspace{2mm}
	
	\noindent{\bfseries Keywords:} Hamiltonicity; traceability; balanced bipartite graph; nearly balanced bipartite graph; spectral analogues; feasible graph parameter
	\par\vspace{2mm}
	
	\noindent{\bfseries 2020 Mathematics Subject Classification:}  05C35; 05C45; 05C50
\end{abstract}

\section{Introduction}

For a graph $G$, let $V(G)$ and $E(G)$ denote the vertex set and edge set of a graph $G$, respectively. Let $N_G(v)$ and $d_G(v)=|N_G(v)|$ denote the neighborhood and the degree of a vertex $v$, respectively. For $S\subseteq V(G)$, let $N_G(S):=\bigcup_{v\in S}N_G(v)$.
Let $\delta(G)$ denote the minimum degree of $G$, and let $e(G)=|E(G)|$ denote its number of edges. 
If $G$ is bipartite with partite sets $X$ and $Y$, we write $G=G[X,Y]$. 
For $U\subseteq X$ and $W\subseteq Y$, let $G[U,W]$ denote the bipartite subgraph consisting of all edges of $G$ between $U$ and $W$. 
We write $K_{a,b}$ for the complete bipartite graph with part sizes $a$ and $b$. A balanced bipartite graph of order $2n$ is a bipartite graph $G=G[X,Y]$ with $|X|=|Y|=n$. Throughout the paper, a nearly balanced bipartite graph of order $2n-1$ means a bipartite graph $G=G[X,Y]$ with $|X|=n$ and $|Y|=n-1$.

Hamiltonian problems are among the central topics in graph theory. A graph $G$ is called Hamiltonian (traceable) if it contains a Hamilton cycle (Hamilton path). From the extremal point of view, a natural inverse problem is to determine the maximum number of edges in a graph with minimum degree at least $k$ that contains no Hamilton cycle or Hamilton path. This line of research goes back to Erd\H{o}s~\cite{Erdos1962}, who determined the maximum number of edges in a non-Hamiltonian graph with minimum degree at least $k$.

For balanced bipartite graphs, Moon and Moser~\cite{MoM1963} proved the following sharp edge-extremal theorem relating the minimum degree and the number of edges for Hamilton cycles.

\begin{thm}[Moon and Moser~\cite{MoM1963}]\label{thm:moon-moser}
Let $G$ be a balanced bipartite graph of order $2n$ with $\delta(G)\ge k$, where
$1\le k\le n/2$. If
\[
e(G)>
\max\left\{
n(n-k)+k^2,\,
n\left(n-\left\lfloor\frac n2\right\rfloor\right)
+\left\lfloor\frac n2\right\rfloor^2
\right\},
\]
then $G$ is Hamiltonian.
\end{thm}

Spectral extremal versions of Hamiltonian problems have been studied
extensively in recent years. Let $A(G)$ and $D(G)$ denote the adjacency
matrix and the degree matrix of $G$, respectively. The spectral radius
$\rho(G)$ is the largest eigenvalue of $A(G)$, and the signless Laplacian
spectral radius $q(G)$ is the largest eigenvalue of $Q(G)=D(G)+A(G)$.

Spectral extremal problems, in which one optimizes an eigenvalue over a prescribed graph class, go back to Brualdi and Solheid~\cite{BrualdiSolheid1986}.
Several related developments illustrate the scope of this spectral
approach. Fiedler and Nikiforov~\cite{FiedlerNikiforov2010} obtained
tight spectral-radius conditions for the existence of Hamilton paths and
cycles in general graphs. Lu, Liu, and Tian~\cite{LuLiuTian2012} gave
related spectral conditions for Hamiltonian graphs, and Ning and
Ge~\cite{NingGe2015} further studied spectral-radius conditions for Hamiltonian properties. Nikiforov~\cite{Nikiforov2016} studied
Hamiltonian paths and cycles under a prescribed large minimum degree. 
For the stronger property of Hamilton-connectedness, Zhou and
Wang~\cite{ZhouWang2017} established sufficient conditions in terms of
the edge number, the adjacency spectral radius, and the signless
Laplacian spectral radius, and also considered traceability from every
vertex. Zhou, Wang, and Lu~\cite{ZhouWangLu2020a,ZhouWangLu2020b} later
established related sufficient conditions in terms of the adjacency
spectral radius and the signless Laplacian spectral radius, including a
large-minimum-degree signless Laplacian version. The same circle of questions has also been extended
beyond ordinary Hamilton cycles: Yan, He, Feng, and Liu~\cite{YanHeFengLiu2023}
studied the spectral radius and the 2-power of a Hamilton cycle; He, Li,
and Feng~\cite{HeLiFeng2024} considered spectral conditions for rainbow
Hamilton paths; and Zhang, Li, Feng, and Liu~\cite{ZhangLiFengLiu2025}
developed signless Laplacian analogues for rainbow Hamilton paths,
matchings, and linear forests. 

Li and Ning~\cite{LiN2016,LiN2017} obtained spectral analogues of the Moon--Moser theorem for Hamiltonicity in balanced bipartite graphs and for traceability in nearly balanced bipartite graphs, determining the sharp adjacency spectral radius and signless Laplacian spectral radius thresholds under the quadratic part-size condition $n\ge (k+1)^2$. 
One of the aims of this paper is to show that the same thresholds already hold in a linear range. 
We use the following notation, following Liu, Ning, and Wang~\cite{LiuNingWang2026}. Let $C_n$ and
$P_n$ denote the cycle and the path on $n$ vertices, respectively. For
$\lambda\in\{\rho,q\}$ and a graph property $\mathcal H$, let
$\spex_\lambda(n,\mathcal H;\delta\ge k)$ denote the maximum value of
$\lambda(G)$ over all $n$-vertex graphs $G$ with $\delta(G)\ge k$
that do not possess $\mathcal H$. The corresponding family of extremal
graphs is denoted by $\SPEX_\lambda(n,\mathcal H;\delta\ge k)$.

To formulate our results in the same extremal language, we recall the following general problem of Liu, Ning, and Wang~\cite{LiuNingWang2026}.

\begin{problem}[Liu, Ning, and Wang~\cite{LiuNingWang2026}]\label{prob:LN-spectral}
	For integers $n\ge3$ and $k\ge1$ with $n\ge2k+1$, determine
	\[
	\spex_\rho(n,C_n;\delta\ge k),\quad
	\spex_q(n,C_n;\delta\ge k),\quad
	\spex_\rho(n,P_n;\delta\ge k),\quad
	\spex_q(n,P_n;\delta\ge k).
	\]
\end{problem}

For two disjoint graphs $H_1$ and $H_2$, let $H_1\cup H_2$ denote their
disjoint union, and let $H_1\vee H_2$ denote their join. Also, $sK_1$
denotes the disjoint union of $s$ copies of $K_1$.

The following theorem records the spectral solution, due to Liu, Ning,
and Wang~\cite{LiuNingWang2026}, for the Hamilton-cycle part
of Problem~\ref{prob:LN-spectral}.

\begin{thm}[Liu, Ning, and Wang~\cite{LiuNingWang2026}]
\label{thm:LNW-spectral-cycle}
Let $k$ and $n$ be positive integers with $n\ge 2k+1$. Then the following
hold:
\begin{enumerate}
    \item[(i)]
    \[
    \spex_{\rho}(n,C_n;\delta\ge k)
    =
    \max\left\{
    \rho\left(K_s\vee(sK_1\cup K_{n-2s})\right):
    k\le s\le \left\lfloor\frac{n-1}{2}\right\rfloor
    \right\},
    \]
    and
    \[
    \SPEX_{\rho}(n,C_n;\delta\ge k)
    \subseteq
    \left\{
    K_s\vee(sK_1\cup K_{n-2s}):
    k\le s\le \left\lfloor\frac{n-1}{2}\right\rfloor
    \right\}.
    \]

    \item[(ii)]
    \[
    \spex_{q}(n,C_n;\delta\ge k)
    =
    \max\left\{
    q\left(K_s\vee(sK_1\cup K_{n-2s})\right):
    k\le s\le \left\lfloor\frac{n-1}{2}\right\rfloor
    \right\},
    \]
    and
    \[
    \SPEX_{q}(n,C_n;\delta\ge k)
    \subseteq
    \left\{
    K_s\vee(sK_1\cup K_{n-2s}):
    k\le s\le \left\lfloor\frac{n-1}{2}\right\rfloor
    \right\}.
    \]
\end{enumerate}
\end{thm}

The path analogues can also be obtained from the feasible-parameter theorem of Liu, Ning, and Wang~\cite{LiuNingWang2026}, although they are
not explicitly listed in Theorem~\ref{thm:LNW-spectral-cycle}. In contrast to Theorem~\ref{thm:LNW-spectral-cycle}, which concerns ordinary graphs, the present paper treats the corresponding Moon--Moser type problem in balanced and nearly balanced bipartite graphs. We prove that the sharp spectral thresholds for Hamiltonicity in balanced bipartite graphs and for traceability in nearly balanced bipartite graphs remain valid in the linear ranges $n\ge 2k$ and $n\ge 2k+1$, respectively. 

The proof strategy of the present paper is inspired by the extremal framework used in \cite{LiuNingWang2026}. To prove our results in a unified way, we use the notion of feasible graph parameters introduced by Ai, Lei, Ning, and Shi \cite{AiL2025}. We recall the relevant operation first. For two vertices $x,y\in V(G)$, the Kelmans operation $G[x\to y]$ is obtained from $G$ by deleting every edge $xz$ with $z\in N_G(x)\setminus N_G(y)$ and adding the corresponding edge $yz$.
Equivalently, the private neighbors of $x$ outside $N_G(y)$ are transferred from $x$ to $y$.

\begin{definition}[Ai, Lei, Ning, Shi~\cite{AiL2025}]\label{def:feasible}
Let $\mathcal P$ be a graph parameter defined on a class of connected
graphs. We call $\mathcal P$ a feasible parameter if it satisfies the
following two properties:
\begin{enumerate}
    \item[(i)] if $xy\notin E(G)$ and $G+xy$ belongs to the class, then
    $\mathcal P(G+xy)>\mathcal P(G)$;
    \item[(ii)] if $xy\notin E(G)$ and $G[x\to y]$ belongs to the class,
    then $\mathcal P(G[x\to y])\ge \mathcal P(G)$.
\end{enumerate}
\end{definition}

Ai, Lei, Ning, and Shi~\cite{AiL2025} showed that the number of edges, the spectral radius, and the signless Laplacian spectral radius are feasible graph parameters. In this paper, we establish feasible-parameter versions of Moon--Moser type extremal theorems for non-Hamiltonian balanced
bipartite graphs and for non-traceable nearly balanced bipartite graphs. These structural results yield the sharp spectral Moon--Moser type theorems for Hamiltonicity and traceability in the linear range.

We shall also use the following notation. For $\lambda\in\{\rho,q\}$, let $\spex_{\lambda}^{\mathrm{B}}(2n,C_{2n};\delta\ge k)$ denote the maximum value of $\lambda(G)$ over all non-Hamiltonian balanced bipartite graphs $G$ of order $2n$ with $\delta(G)\ge k$.
The corresponding extremal family is denoted by $\SPEX_{\lambda}^{\mathrm{B}}(2n,C_{2n};\delta\ge k)$. For nearly balanced bipartite graphs, let $\spex_{\lambda}^{\mathrm{NB}}(2n-1,P_{2n-1};\delta\ge k)$ denote the maximum value of $\lambda(G)$ over all non-traceable nearly
balanced bipartite graphs $G=G[X,Y]$ with $|X|=n$, $|Y|=n-1$, and
$\delta(G)\ge k$. The corresponding extremal family is denoted by $\SPEX_{\lambda}^{\mathrm{NB}}(2n-1,P_{2n-1};\delta\ge k)$. When an ambient graph class is specified, we write $\ex_{\mathcal P}(n,\mathcal H;\delta\ge k)$ for the maximum value of $\mathcal P(G)$ over all graphs $G$ in that class with $|V(G)|=n$ and $\delta(G)\ge k$ that do not possess $\mathcal H$. The corresponding
family of extremal graphs is denoted by
$\EX_{\mathcal P}(n,\mathcal H;\delta\ge k)$.

\subsection{Hamiltonicity for balanced bipartite graphs}

For $1\le s\le \lfloor n/2\rfloor$, let $B_n^s$ be the balanced bipartite graph with
bipartition $X=X_1\cup X_2$ and $Y=Y_1\cup Y_2$, where 
$$|X_1|=|Y_1|=s, |X_2|=|Y_2|=n-s,$$ and
$$E(B_n^s)=\{xy:x\in X_1,\,y\in Y_{1}\}\cup \{xy:x\in X_{2},\,y\in Y\}.$$
Equivalently, $B_n^s$ is obtained from $K_{n,n}$ by deleting all edges between a fixed $s$-set in one part and a fixed $(n-s)$-set in the other part.

Li and Ning~\cite{LiN2016} proved the following spectral analogue of the Moon--Moser theorem.

\begin{thm}[Li and Ning~\cite{LiN2016}]\label{thm:LiN-hamiltonian-cycle}
	Let $G$ be a balanced bipartite graph of order $2n$ with $\delta(G)\ge k\ge1$. If
	$n\ge (k+1)^2$, then the following statements hold.
	\begin{enumerate}
		\item[(i)] If $\rho(G)\ge \rho(B_n^k)$, then $G$ is Hamiltonian unless
		$G\cong B_n^k$.
		\item[(ii)] If $q(G)\ge q(B_n^k)$, then $G$ is Hamiltonian unless
		$G\cong B_n^k$.
	\end{enumerate}
\end{thm}

Our first structural result is a feasible-parameter
version of the Moon--Moser theorem. 

\begin{thm}\label{th:ch-1.5}
Let $1\le k\le \lfloor n/2\rfloor$. Let $\mathcal{P}$ be a graph parameter defined on all balanced bipartite graphs of order $2n$. Assume that on connected members of this class $\mathcal{P}$ is feasible, and assume further that $\mathcal{P}(H)<\mathcal{P}(G)$ whenever $H$ is a proper spanning subgraph of a connected graph $G$ in the class. Then
\[
\ex_{\mathcal P}(2n,C_{2n};\delta\ge k)
=
\max\left\{
\mathcal P(B_n^s):
k\le s\le \left\lfloor\frac n2\right\rfloor
\right\}.
\]
Moreover,
\[
\EX_{\mathcal P}(2n,C_{2n};\delta\ge k)
\subseteq
\left\{
B_n^s: k\le s\le \left\lfloor\frac n2\right\rfloor
\right\}.
\]
\end{thm}

We reduce the quadratic condition $n\ge (k+1)^2$ in Theorem~\ref{thm:LiN-hamiltonian-cycle} to a linear range $n\ge 2k$, and obtain the following theorem. 
\begin{thm}\label{th:ch-1.6}
Let $k\ge 1$ and $n\ge 2k$. Then
\[
\SPEX_{\rho}^{\mathrm B}(2n,C_{2n};\delta\ge k)
=
\{B_{n}^{k}\}
\quad\text{and}\quad
\spex_{\rho}^{\mathrm B}(2n,C_{2n};\delta\ge k)=\rho(B_n^k),
\]
and
\[
\SPEX_{q}^{\mathrm B}(2n,C_{2n};\delta\ge k)
=
\{B_{n}^{k}\}
\quad\text{and}\quad
\spex_{q}^{\mathrm B}(2n,C_{2n};\delta\ge k)=q(B_n^k).
\]
\end{thm}

\subsection{Traceability for nearly balanced bipartite graphs}

For integers $n\ge 2$ and $1\le s\le n-1$, let $S_n^s$ be the
nearly balanced bipartite graph with bipartition
$X=X_1\cup X_2$ and $Y=Y_1\cup Y_2$, where
\[
|X_1|=s,\quad |X_2|=n-s,\quad |Y_1|=n-s-1,\quad |Y_2|=s,
\]
and
\[
E(S_n^s)=\{xy:x\in X_1,\ y\in Y\}\cup
\{xy:x\in X_2,\ y\in Y_1\}.
\]
Here $Y_1$ is allowed to be empty when $s=n-1$. 

For integers $n\ge 2$ and $0\le t\le \lfloor (n-2)/2\rfloor$, let
$T_n^t$ be the nearly balanced bipartite graph with bipartition
$X=X_1\cup X_2$ and $Y=Y_1\cup Y_2$, where
$$
|X_1|=n-t-1,\quad |X_2|=t+1,\quad |Y_1|=t,\quad |Y_2|=n-t-1,
$$
and
$$
E(T_n^t)=\{xy:x\in X_1,\ y\in Y\}\cup
\{xy:x\in X_2,\ y\in Y_1\}.
$$
Here $Y_1$ is allowed to be empty when $t=0$. With this convention,
$$
T_n^t\cong S_n^{\,n-t-1}.
$$

Li and Ning~\cite{LiN2017} gave the
following spectral result for nearly balanced bipartite graphs.

\begin{thm}[Li and Ning \cite{LiN2017}]
    Let $G$ be a nearly balanced bipartite graph on $2n-1$ vertices, with minimum degree $\delta(G)\geq k$, where $k\geq 1$ and $n\geq (k+1)^{2}$.

    \begin{itemize}
        \item [(i)] If $\rho(G)\geq \rho(S^{k}_{n})$, then $G$ is traceable unless $G\cong S^{k}_{n}$.
        \item[(ii)] If $q(G)\geq q(S^{k}_{n})$, then $G$ is traceable unless $G\cong S^{k}_{n}$. 
    \end{itemize}
\end{thm}

For $1\le s\le n-1$, let $\mathcal S_{n}^{s}$ be the family of all connected non-traceable nearly balanced bipartite graphs $G=G[X,Y]$ with $|X|=n$ and $|Y|=n-1$ for which there exists a subset $Y_2\subseteq Y$ with $|Y_2|=s$ such that, putting $Y_1=Y\setminus Y_2$, we have $G[X,Y_1]=K_{n,n-s-1}$ and $d_G(y)=s$ for every $y\in Y_2$.

For $0\le t\le \left\lfloor (n-2)/2\right\rfloor$, let $\mathcal T_{n}^{t}$ be the family of all connected non-traceable
nearly balanced bipartite graphs $G=G[X,Y]$ with $|X|=n$ and
$|Y|=n-1$ for which there exists a subset $X_2\subseteq X$ with
$|X_2|=t+1$ such that, putting $X_1=X\setminus X_2$, we have $G[X_1,Y]=K_{n-t-1,n-1}$ and $d_G(x)=t$ for every $x\in X_2$. 

\begin{thm}\label{th:ch-1.8}
Let $n\ge 2k+1$ and $k\ge1$. Let $\mathcal{P}$ be a graph parameter defined on all nearly balanced bipartite graphs $G=G[X,Y]$ with $|X|=n$ and $|Y|=n-1$. Assume that on connected members of this class $\mathcal{P}$ is feasible, and assume further that $\mathcal{P}(H)<\mathcal{P}(G)$ whenever $H$ is a proper spanning subgraph of a connected graph $G$ in the class. Then
\[
\ex_{\mathcal P}(2n-1,P_{2n-1};\delta\ge k)
=
\max\left(
\left\{
\mathcal P(S_{n}^{s}):
k\le s\le \left\lfloor\frac n2\right\rfloor
\right\}
\cup
\left\{
\mathcal P(T_{n}^{t}):
k\le t\le \left\lfloor\frac{n-2}{2}\right\rfloor
\right\}
\right).
\]
Moreover,
$$
\EX_{\mathcal P}(2n-1,P_{2n-1};\delta\ge k)
\subseteq
\left(
\bigcup_{k\le s\le \left\lfloor n/2\right\rfloor}
\mathcal S_n^s
\right)
\cup
\left(
\bigcup_{k\le t\le \left\lfloor (n-2)/2\right\rfloor}
\mathcal T_n^t
\right).
$$
Here the second set and the second union are understood to be empty if $k>\left\lfloor (n-2)/2\right\rfloor$.
\end{thm}

As a consequence of the feasible-parameter theorem and the spectral comparison of the extremal graphs, we obtain the following sharp result.

\begin{thm}\label{th:ch-1.9}
Let $k\ge 1$ and $n\ge 2k+1$. Then
\[
\SPEX_{\rho}^{\mathrm {NB}}(2n-1,P_{2n-1};\delta\ge k)
=
\{S_{n}^{k}\}
\quad\text{and}\quad
\spex_{\rho}^{\mathrm {NB}}(2n-1,P_{2n-1};\delta\ge k)=\rho(S_{n}^{k}),
\]
and
\[
\SPEX_{q}^{\mathrm {NB}}(2n-1,P_{2n-1};\delta\ge k)
=
\{S_{n}^{k}\}
\quad\text{and}\quad
\spex_{q}^{\mathrm {NB}}(2n-1,P_{2n-1};\delta\ge k)=q(S_{n}^{k}).
\]
\end{thm}

The remainder of the paper is organized as follows. Section 2 introduces the preliminary tools used throughout the paper. In Section 3, we prove the feasible-parameter extremal theorem for non-Hamiltonian balanced bipartite graphs and derive the corresponding sharp spectral Hamiltonicity results. In Section 4, we treat traceability in nearly balanced bipartite graphs and prove the corresponding feasible-parameter extremal
theorems together with their spectral consequences.

\section{Preliminaries}

In this section, we collect several preliminary lemmas used in the proofs of the main results.

We begin with the notion of bipartite closure, which is a basic tool in the study of Hamilton cycles in balanced bipartite graphs.

For a balanced bipartite graph $G=G[X,Y]$ with $|X|=|Y|=n$, let $\operatorname{cl}_B(G)$ denote the bipartite closure of $G$, obtained by repeatedly adding an edge $xy$ between nonadjacent vertices $x\in X$ and $y\in Y$ whenever their degree sum in the current graph is at least $n+1$, until no such pair of vertices remains.

\begin{thm}[Bondy and Chv\'atal \cite{BoC1976}]\label{th:ch-2.1}
Let $G$ be a balanced bipartite graph of order $2n$. Then $G$ is Hamiltonian
if and only if $\operatorname{cl}_B(G)$ is Hamiltonian.
\end{thm}

We recall the standard quotient-matrix method for equitable partitions, which will be
used in the spectral computations below. Let $M$ be a matrix whose rows and columns
are indexed by a finite set $V$, and let
\[
\Pi=\{V_1,V_2,\ldots,V_m\}
\]
be a partition of $V$. If, for every $1\le i,j\le m$, each row of the block
$M[V_i,V_j]$ has the same sum $b_{ij}$, then $\Pi$ is called an equitable
partition of $M$, and the matrix $B=(b_{ij})$ is called the quotient matrix
of $M$ with respect to $\Pi$.

\begin{lem}[Brouwer and Haemers~\cite{BrouwerHaemers2012}]\label{le:ch-2.2}
	Let $M$ be a nonnegative irreducible symmetric matrix, and let $\Pi$ be an equitable
	partition of $M$ with quotient matrix $B$. Then every eigenvalue of $B$ is an eigenvalue of $M$. Moreover, the Perron eigenvalue of $B$ is equal to the Perron eigenvalue of $M$. Here the Perron eigenvalue of a nonnegative irreducible matrix means its positive eigenvalue whose modulus is maximal.
\end{lem}

\begin{lem}\label{le:ch-2.3}
Let $G=G[X,Y]$ be a bipartite graph. Let $U\subseteq X$ and
$W\subseteq Y$ with $|W|=r$. Suppose that $d_G(u)\le r$ for every $u\in U$. Then there exists a graph $\Gamma$ obtained from $G$ by a finite sequence of Kelmans operations inside the part $Y$ such that $N_\Gamma(U)\subseteq W$ and $d_\Gamma(u)=d_G(u)$ for every $u\in U$. Consequently, if a graph parameter $\mathcal P$ is nondecreasing under the Kelmans operations used in this sequence, then $\mathcal P(G)\le \mathcal P(\Gamma)$.
\end{lem}

\begin{proof}
Let $\Gamma$ be the current graph, initially equal to $G$. We shall
repeatedly apply Kelmans operations inside the part $Y$. During the process,
the graph remains bipartite with the same bipartition $X,Y$.

Define $\Phi(\Gamma)=\sum_{u\in U}|N_\Gamma(u)\setminus W|$. If $\Phi(\Gamma)=0$, then $N_\Gamma(U)\subseteq W$, and we are done.

Suppose that $\Phi(\Gamma)>0$. Then there exist $u\in U$ and
$y\in Y\setminus W$ such that $uy\in E(\Gamma)$. Since every Kelmans
operation inside $Y$ preserves the degrees of vertices in $X$, we have
\[
d_\Gamma(u)=d_G(u)\le r=|W|.
\]
As $u$ has a neighbor outside $W$, it cannot be adjacent to all vertices
of $W$. Hence there exists a vertex $w\in W$ such that $uw\notin E(\Gamma)$.

Now replace $\Gamma$ by $\Gamma[y\to w]$. Since $y$ and $w$ lie in the
same part $Y$, this is a Kelmans operation inside $Y$. For every vertex
$x\in X$ adjacent to $y$ but not to $w$, the operation deletes the edge
$xy$ and adds the edge $xw$. Thus the degree of every vertex of $X$, and
in particular of every vertex of $U$, is preserved.

Moreover, this operation creates no new edge from a vertex of $U$ to $Y\setminus W$; it only replaces edges incident with $y\in Y\setminus W$ by corresponding edges incident with $w\in W$. Since the chosen edge $uy$ is moved to $uw$, the
quantity $\Phi(\Gamma)$ strictly decreases.

The quantity $\Phi(\Gamma)$ is a nonnegative integer. Therefore the process
terminates after finitely many steps. At the end, we obtain a graph
$\Gamma$ with $N_\Gamma(U)\subseteq W$. Since every operation preserves
the degrees of vertices in $X$, we also have $d_\Gamma(u)=d_G(u)$ for every $u\in U$.

Finally, if $\mathcal{P}$ is nondecreasing under the Kelmans operations used above,
then applying this monotonicity at each step gives $\mathcal{P}(G)\le \mathcal{P}(\Gamma)$.
This completes the proof.
\end{proof}

\section{Proofs for Hamiltonicity of balanced bipartite graphs}
In this section, we prove Theorems~\ref{th:ch-1.5} and~\ref{th:ch-1.6}. We first establish one lemma for balanced bipartite graphs.

\begin{lem}\label{le:ch-3.1}
Let $G=G[X,Y]$ be a balanced bipartite graph with $\delta(G)\geq k$ and $|X|=|Y|=n$, where $k\geq 1$ and $n\geq 2$. If $G$ has no Hamilton cycle, then there exists an integer $s$ with $k\le s\le \left\lfloor \frac n2 \right\rfloor$ such that one of $X$ and $Y$ contains $s$ vertices of degree at most $s$.
\end{lem}

\begin{proof}
Let $H=\operatorname{cl}_B(G)$. Since $G$ has no Hamilton cycle,
Theorem~\ref{th:ch-2.1} implies that $H$ has no
Hamilton cycle. Hence $H\ne K_{n,n}$.

Since $H\ne K_{n,n}$, choose $x\in X$ and
$y\in Y$ with $xy\notin E(H)$ such that
$d_{H}(x)+d_{H}(y)$ is maximum. By symmetry, we may assume that
$d_H(x)\le d_H(y)$.

Since $xy\notin E(H)$, the definition of the bipartite closure gives $d_H(x)+d_H(y)\le n$. Put $s=d_H(x)$. Since $H$ is obtained from $G$ by adding edges, we have
\[
s=d_H(x)\ge d_G(x)\ge \delta(G)\ge k.
\]
Moreover, from $d_H(x)\le d_H(y)$ and $d_H(x)+d_H(y)\le n$, we get
\[
2s=2d_H(x)\le n,
\]
and hence $s\le \left\lfloor \frac n2\right\rfloor$.

It remains to find $s$ vertices of degree at most $s$. Since
$d_H(y)\le n-s$, we have
\[
|X\setminus N_H(y)|=n-d_H(y)\ge s.
\]
For every $x'\in X\setminus N_H(y)$, we have $x'y\notin E(H)$. By the
maximal choice of the non-edge $xy$,
\[
d_H(x')+d_H(y)\le d_H(x)+d_H(y),
\]
and therefore $d_H(x')\le d_H(x)=s$. Thus $X\setminus N_H(y)$ contains at least $s$ vertices whose degrees in $H$ are at most $s$. Since $G$ is a subgraph of $H$, these vertices also have degree at most $s$ in $G$. Hence $X$ contains $s$ vertices of degree at most $s$ in $G$.
This proves the lemma.
\end{proof}

\begin{proof}[Proof of Theorem~\ref{th:ch-1.5}]
Let $G=G[X,Y]$ be a non-Hamiltonian balanced bipartite graph of order
$2n$ with $\delta(G)\ge k$ such that
\[
\mathcal P(G)=\operatorname{ex}_{\mathcal P}(2n,C_{2n};\delta\ge k).
\]

By Lemma~\ref{le:ch-3.1}, there exists an integer $s$ with $k\le s\le \left\lfloor\frac n2\right\rfloor$ such that one of $X$ and $Y$ contains $s$ vertices of degree at most $s$.
By symmetry, we may assume that there is a set $S=\{x_1,x_2,\ldots,x_s\}\subseteq X$ such that $d_G(x_i)\le s$ for every $i=1,\ldots,s$.

Choose an arbitrary subset $W\subseteq Y$ with $|W|=s$. Put
\[
X_1=S,\qquad X_2=X\setminus S,\qquad
Y_1=W,\qquad Y_2=Y\setminus W.
\]
Then
\[
|X_1|=|Y_1|=s,\qquad |X_2|=|Y_2|=n-s.
\]

Let $H$ be the graph obtained from $G$ by adding all missing edges
between $X_2$ and $Y$. Since $X_2\ne\emptyset$ and every vertex of $S$ has degree at least $k\ge1$, the graph $H$ is connected.
Then the degrees of the vertices in $S$ are
unchanged, and hence $d_H(x_i)=d_G(x_i)\le s$ for every $i=1,\ldots,s$. If $G\ne H$, then $G$ is a proper spanning subgraph of the connected graph $H$, and hence $\mathcal P(G)<\mathcal P(H)$ by the additional assumption. Thus $\mathcal P(G)\le \mathcal P(H)$.

By Lemma~\ref{le:ch-2.3}, applied to $H[X,Y]$ with $U=S$ and $W=Y_1$,
there exists a graph $\Gamma$ obtained from $H$ by a finite sequence of
Kelmans operations inside the part $Y$ such that $N_\Gamma(S)\subseteq Y_1$ and $d_\Gamma(x_i)=d_H(x_i)=d_G(x_i)$ for every $i=1,\ldots,s$.

Moreover, all graphs arising in this Kelmans sequence are connected.
Indeed, the complete bipartite subgraph $H[X_2,Y]=K_{n-s,n}$ is preserved throughout the Kelmans operations, and every vertex of $S$ remains adjacent to some vertex of $Y$. Hence the Kelmans monotonicity of $\mathcal P$ gives $\mathcal P(H)\le \mathcal P(\Gamma)$. Since $N_\Gamma(S)\subseteq Y_{1}$, the graph $\Gamma$ has no edge between $X_1$ and $Y_2$. Hence $\Gamma$ is a subgraph of $B_n^s$ with respect to the above partition. If $\Gamma\ne B_n^s$, then $\Gamma$ is a proper spanning subgraph of the
connected graph $B_n^s$, and hence $\mathcal P(\Gamma)<\mathcal P(B_n^s)$. Thus $\mathcal P(\Gamma)\le \mathcal P(B_n^s)$. Consequently,
\[
\mathcal P(G)
\le \mathcal P(H)
\le \mathcal P(\Gamma)
\le \mathcal P(B_n^s)
\le
\max\left\{
\mathcal P(B_n^t):
k\le t\le \left\lfloor\frac n2\right\rfloor
\right\}.
\]
Thus
\[
\operatorname{ex}_{\mathcal P}(2n,C_{2n};\delta\ge k)
\le
\max\left\{
\mathcal P(B_n^t):
k\le t\le \left\lfloor\frac n2\right\rfloor
\right\}.
\]

Conversely, for every integer $t$ with $k\le t\le \left\lfloor\frac n2\right\rfloor$, the graph $B_n^t$ satisfies $\delta(B_n^t)=t\ge k$. 
Moreover, $B_n^t$ is non-Hamiltonian. Indeed, deleting the $t$ vertices of $Y_1$ leaves the $t$ vertices of $X_1$ as isolated vertices together with one further component. Thus the remaining graph has $t+1$ components after deleting $t$ vertices, which is impossible for a Hamiltonian graph.

Therefore every graph $B_n^t$ in the indicated range is admissible in
the extremal problem. Hence
\[
\ex_{\mathcal P}(2n,C_{2n};\delta\ge k)
\ge
\max\left\{
\mathcal P(B_n^t):
k\le t\le \left\lfloor\frac n2\right\rfloor
\right\}.
\]
Combining the two inequalities gives
\[
\ex_{\mathcal P}(2n,C_{2n};\delta\ge k)
=
\max\left\{
\mathcal P(B_n^t):
k\le t\le \left\lfloor\frac n2\right\rfloor
\right\}.
\]

It remains to determine the extremal graphs. Suppose that $G$ is
extremal. Then
\[
\mathcal P(G)=
\max\left\{
\mathcal P(B_n^t):
k\le t\le \left\lfloor\frac n2\right\rfloor
\right\}.
\]
Hence all inequalities in
\[
\mathcal P(G)
\le \mathcal P(H)
\le \mathcal P(\Gamma)
\le \mathcal P(B_n^s)
\le
\max\left\{
\mathcal P(B_n^t):
k\le t\le \left\lfloor\frac n2\right\rfloor
\right\}
\]
must be equalities.

Since $G$ is a spanning subgraph of the connected graph $H$, the equality
$\mathcal P(G)=\mathcal P(H)$ and the additional assumption force $G=H$.
Hence $G[X\setminus S,Y]=K_{n-s,n}$.

Moreover, the equality $\mathcal P(\Gamma)=\mathcal P(B_n^s)$ forces $\Gamma=B_n^s$, because $B_n^s$ is obtained from $\Gamma$ by adding edges and $\mathcal P$ is strictly increasing under edge addition. All Kelmans operations used above were performed inside the part $Y$, so the degrees of vertices in $X$ are preserved. Since $\Gamma=B_n^s$, every vertex of $S$ has degree $s$ in $G$.

We claim that $G\cong B_n^s$. Suppose not. Since every vertex of $S$ has
degree $s$ and $G[X\setminus S,Y]=K_{n-s,n}$, the assumption
$G\not\cong B_n^s$ implies $|N_G(S)|\ge s+1$. Indeed, if $|N_G(S)|=s$, then every vertex of $S$ has degree $s$ and
therefore has the same neighborhood $N_G(S)$, which gives
$G\cong B_n^s$.

Let $y\in N_G(S)$. Since $y$ is adjacent to every vertex of
$X\setminus S$ and to at least one vertex of $S$, we have $d_G(y)\ge n-s+1$. Thus, for every non-edge $xy$ with $x\in S$ and $y\in N_G(S)$,
\[
d_G(x)+d_G(y)\ge s+(n-s+1)=n+1.
\]
Hence all missing edges between $S$ and $N_G(S)$ are added in the
bipartite closure $\operatorname{cl}_B(G)$.

After these edges are added, every vertex $x\in S$ has degree at least $|N_G(S)|\ge s+1$ in the closure. On the other hand, for every $y\in Y\setminus N_G(S)$, we have $d_G(y)=n-s$, because $y$ is adjacent to every vertex of $X\setminus S$ and to no
vertex of $S$. Therefore
\[
d_{\operatorname{cl}_B(G)}(x)+d_{\operatorname{cl}_B(G)}(y)
\ge (s+1)+(n-s)=n+1.
\]
It follows that all missing edges between $S$ and $Y\setminus N_G(S)$
are also added in the bipartite closure. Consequently, $\operatorname{cl}_B(G)=K_{n,n}$. Thus $\operatorname{cl}_B(G)$ has a Hamilton cycle. By
Theorem~\ref{th:ch-2.1}, the graph $G$ has a Hamilton cycle, a
contradiction. Hence $G\cong B_n^s$.

Therefore
\[
\EX_{\mathcal P}(2n,C_{2n};\delta\ge k)
\subseteq
\left\{
B_n^s:
k\le s\le \left\lfloor\frac n2\right\rfloor
\right\}.
\]
This completes the proof.
\end{proof}

\begin{lem}\label{le:ch-3.2}
Let $k\geq 1$ and $n\geq 2k$. Then
\[
\max_{k\le s\le \lfloor n/2\rfloor}\rho(B_n^s)=\rho(B_n^k)
\quad\text{and}\quad
\max_{k\le s\le \lfloor n/2\rfloor}q(B_n^s)=q(B_n^k).
\]
Moreover, in both maxima equality holds only when $s=k$.
\end{lem}

\begin{proof}
Put $m=\lfloor n/2\rfloor$ and
\[
u_s=s(n-s).
\]
Then $u_s$ is strictly increasing for $1\le s\le m$.

Let $X=X_1\cup X_2$ and $Y=Y_1\cup Y_2$ be the bipartition of
$B_n^s$, where $|X_1|=|Y_1|=s$ and $|X_2|=|Y_2|=n-s$.
With respect to the equitable partition $X_1, X_2, Y_1, Y_2$, the adjacency quotient matrix of $B_n^s$ is
\[
B_A(s)=
\begin{pmatrix}
0 & 0 & s & 0\\
0 & 0 & s & n-s\\
s & n-s & 0 & 0\\
0 & n-s & 0 & 0
\end{pmatrix}.
\]
Since the partition is equitable, $\rho(B_n^s)$ is the Perron eigenvalue of
$B_A(s)$. A direct calculation gives
$$
\det(xI-B_A(s))
=
x^4-(n^2-u_s)x^2+u_s^2.
$$
Thus $\rho(B_n^s)$ satisfies
$$
\rho(B_n^s)^4-(n^2-u_s)\rho(B_n^s)^2+u_s^2=0.
$$
Equivalently, $\rho(B_n^s)^2$ is a root of
$$
z^2-(n^2-u_s)z+u_s^2=0.
$$
The two roots of this equation give the possible values of $x^2$ for the
eigenvalues $x$ of $B_A(s)$. Since $\rho(B_n^s)$ is the Perron eigenvalue of $B_A(s)$, $\rho(B_n^s)^2$ is the larger root. Hence $\rho(B_n^s)^2=\phi(u_s)$, where
$$
\phi(u)= \frac{n^2-u+\sqrt{(n^2-3u)(n^2+u)}}{2}.
$$
For $0\le u\le n^2/4$, we have
$$
\phi'(u)
=
-\frac12-\frac{n^2+3u}{2\sqrt{(n^2-3u)(n^2+u)}}<0.
$$
Thus $\rho(B_n^s)^2$, and hence $\rho(B_n^s)$, is strictly decreasing in $s$
for $1\le s\le m$. It follows that
\[
\max_{k\le s\le m}\rho(B_n^s)=\rho(B_n^k),
\]
and equality holds only when $s=k$.

Next we consider the signless Laplacian spectral radius. With respect to
the same equitable partition, the signless Laplacian quotient matrix is
\[
B_Q(s)=
\begin{pmatrix}
s & 0 & s & 0\\
0 & n & s & n-s\\
s & n-s & n & 0\\
0 & n-s & 0 & n-s
\end{pmatrix}.
\]
Indeed, the degrees of vertices in $X_1,X_2,Y_1,Y_2$ are respectively
$s,n,n,n-s$. Hence, by the quotient-matrix method, $q(B_n^s)$ is the largest eigenvalue
of $B_Q(s)$. 
A direct calculation gives
$$
\det(xI-B_Q(s))
=
x(x-n)\bigl(x^2-2nx+2u_s\bigr).
$$
Thus the eigenvalues of $B_Q(s)$ are $0$, $n$, and $n\pm\sqrt{n^2-2u_s}$. Since $u_s\le n^2/4$, the largest one is $n+\sqrt{n^2-2u_s}$. Therefore
$$
q(B_n^s)=n+\sqrt{n^2-2u_s}.
$$
Since $u_s$ is strictly increasing for $1\le s\le m$, the value
$q(B_n^s)$ is strictly decreasing in $s$ for $1\le s\le m$. Hence
\[
\max_{k\le s\le m}q(B_n^s)=q(B_n^k),
\]
and equality holds only when $s=k$. This completes the proof.
\end{proof}

\begin{proof}[Proof of Theorem~\ref{th:ch-1.6}]
Let $\lambda\in\{\rho,q\}$. Since $n\ge 2k$, we have
$1\le k\le \lfloor n/2\rfloor$. Both $\rho$ and $q$ are feasible graph
parameters. Moreover, for $\mathcal P=\rho$ or $\mathcal P=q$, the condition $\mathcal P(H)<\mathcal P(G)$ whenever $H$ is a proper spanning subgraph of a connected graph $G$ follows from the Perron--Frobenius Theorem. By applying Theorem~\ref{th:ch-1.5} to $\mathcal P=\lambda$, we obtain
\[
\spex_{\lambda}^{\mathrm B}(2n,C_{2n};\delta\ge k)
=
\max\left\{
\lambda(B_n^s):
k\le s\le \left\lfloor\frac n2\right\rfloor
\right\}.
\]
By Lemma~\ref{le:ch-3.2}, the maximum on the right-hand side is
attained uniquely at $s=k$. Hence
\[
\spex_{\lambda}^{\mathrm B}(2n,C_{2n};\delta\ge k)
=
\lambda(B_n^k).
\]

It remains to identify the extremal graphs. Let $G\in \SPEX_{\lambda}^{\mathrm B}(2n,C_{2n};\delta\ge k)$. By the equality case of Theorem~\ref{th:ch-1.5}, we have $G\cong B_n^s$ for some integer $s$ with $k\le s\le \left\lfloor\frac n2\right\rfloor$. Since $\lambda(G)=\lambda(B_n^k)$, the uniqueness of the maximum in Lemma~\ref{le:ch-3.2} forces
$s=k$. Thus $G\cong B_n^k$. Conversely, $B_n^k$ is non-Hamiltonian, satisfies $\delta(B_n^k)=k$, and attains the value $\lambda(B_n^k)$. Therefore
\[
\SPEX_{\lambda}^{\mathrm B}(2n,C_{2n};\delta\ge k)
=
\{B_{n}^{k}\}.
\]

Taking $\lambda=\rho$ gives
\[
\SPEX_{\rho}^{\mathrm B}(2n,C_{2n};\delta\ge k)
=
\{B_{n}^{k}\}
\quad\text{and}\quad
\spex_{\rho}^{\mathrm B}(2n,C_{2n};\delta\ge k)
=
\rho(B_n^k).
\]
Taking $\lambda=q$ gives
\[
\SPEX_{q}^{\mathrm B}(2n,C_{2n};\delta\ge k)
=
\{B_{n}^{k}\}
\quad\text{and}\quad
\spex_{q}^{\mathrm B}(2n,C_{2n};\delta\ge k)
=
q(B_n^k).
\]
This completes the proof.
\end{proof}

\section{Proofs for traceability of nearly balanced bipartite graphs}

\begin{lem}\label{le:ch-4.1}
Let $n\geq 2$ and $k\geq 1$, and let $G=G[X,Y]$ be a nearly balanced
bipartite graph with $|X|=n$ and $|Y|=n-1$. Suppose that
$\delta(G)\ge k$. If $G$ has no Hamilton path, then one of the following
holds.

\begin{itemize}
\item[(i)] There exists an integer $s$ with $k\le s\le \left\lfloor\frac n2\right\rfloor$ such that $Y$ contains $s$ vertices of degree at most $s$.

\item[(ii)] There exists an integer $t$ with $k\le t\le \left\lfloor\frac{n-2}{2}\right\rfloor$ such that $X$ contains $t+1$ vertices of degree at most $t$.
\end{itemize}
\end{lem}

\begin{proof}
Add a new vertex $y'$ and put $Y'=Y\cup\{y'\}$. Let $G'$ be the balanced
bipartite graph with bipartition $(X,Y')$ obtained from $G$ by adding all
edges between $y'$ and the vertices of $X$.

We first observe that $G$ has a Hamilton path if and only if $G'$ has a
Hamilton cycle. Indeed, if $G'$ has a Hamilton cycle, then deleting the
vertex $y'$ from this cycle gives a Hamilton path of $G$. Conversely, if
$G$ has a Hamilton path, then, since $|X|=|Y|+1$, the two endvertices of
this path both lie in $X$. Adding the vertex $y'$ and the two edges from $y'$ to the endvertices of the path gives a Hamilton cycle of $G'$.

Since $G$ has no Hamilton path, the graph $G'$ has no Hamilton cycle.
Moreover, $G'$ satisfies $\delta(G')\ge k$, since the degrees of vertices
in $Y$ are unchanged, the degrees of vertices in $X$ increase by one, and
$d_{G'}(y')=n\ge k$.
By Lemma~\ref{le:ch-3.1}, applied to the balanced bipartite graph $G'$,
there exists an integer $r$ with $k\le r\le \left\lfloor\frac n2\right\rfloor$ such that one of the two partite sets $X$ and $Y'$ contains $r$ vertices whose degrees in $G'$ are at most $r$.

Suppose first that $Y'$ contains such a set of $r$ vertices. Since
$d_{G'}(y')=n>r$, the vertex $y'$ is not contained in this set. Hence the
set is contained in $Y$. Moreover, the degrees of vertices of $Y$ are the
same in $G$ and in $G'$. Thus $Y$ contains $r$ vertices of degree at most
$r$ in $G$. Taking $s=r$ gives case {\rm (i)}.

Now suppose that $X$ contains $r$ vertices whose degrees in $G'$ are at
most $r$. For every $x\in X$, the vertex $x$ is adjacent to $y'$ in
$G'$, and therefore $d_G(x)=d_{G'}(x)-1$. Hence these $r$ vertices of $X$ have degree at most $r-1$ in $G$. Put
$t=r-1$. Then $X$ contains $t+1$ vertices of degree at most $t$. Since $r\le \left\lfloor\frac n2\right\rfloor$, we have $t=r-1\leq\left\lfloor\frac{n-2}{2}\right\rfloor$. Moreover, since $\delta(G)\ge k$, the existence of a vertex of degree at
most $t$ implies $t\ge k$. This gives case {\rm (ii)}. This completes the proof.
\end{proof}

\begin{proof}[Proof of Theorem~\ref{th:ch-1.8}]
Set
$$
M=
\max\left(
\left\{\mathcal P(S_n^s): k\le s\le \left\lfloor \frac n2\right\rfloor\right\}
\cup
\left\{\mathcal P(T_n^t): k\le t\le \left\lfloor \frac{n-2}{2}\right\rfloor\right\}
\right).
$$

Let $G=G[X,Y]$ be an extremal non-traceable nearly balanced bipartite
graph of order $2n-1$ with $\delta(G)\ge k$ such that
$$
\mathcal P(G)=
\ex_{\mathcal P}(2n-1,P_{2n-1};\delta\ge k),
$$
where $|X|=n$ and $|Y|=n-1$. By
Lemma~\ref{le:ch-4.1}, one of the following two cases occurs.

\medskip
\noindent\textbf{Case 1.} There exists an integer $s$ with $k\le s\le \left\lfloor n/2\right\rfloor$ and a set $Y_2=\{y_1,\ldots,y_s\}\subseteq Y$ such that $d_G(y_i)\le s$ for every $i=1,\ldots,s$.

Choose an arbitrary subset $X_1\subseteq X$ with $|X_1|=s$, and put
$$
X_2=X\setminus X_1,\qquad Y_1=Y\setminus Y_2.
$$
Then
$$
|X_1|=s,\quad |X_2|=n-s,\quad |Y_1|=n-s-1,\quad |Y_2|=s.
$$

Let $H$ be the graph obtained from $G$ by adding all missing edges between
$X$ and $Y_1$. Then $H[X,Y_1]=K_{n,n-s-1}$. Since $Y_1\ne\emptyset$ and
every vertex of $Y_2$ remains adjacent to some vertex of $X$, the graph $H$ is connected.

The degrees of the vertices in $Y_2$ are unchanged, and hence $
d_H(y_i)=d_G(y_i)\le s$ for every $i=1,\ldots,s$. If $G\ne H$, then $G$ is a proper spanning subgraph of the connected graph $H$, and hence the additional assumption in the theorem gives $\mathcal P(G)<\mathcal P(H)$. Thus $\mathcal P(G)\le \mathcal P(H)$, with equality only if $G=H$.

Apply Lemma~\ref{le:ch-2.3} to $H$, viewed as a bipartite graph with parts $Y$ and $X$, with $U=Y_2$ and $W=X_1$. Then there exists a graph $\Gamma$ obtained from $H$ by a finite sequence of Kelmans operations inside the part $X$ such that $N_\Gamma(Y_2)\subseteq X_1$ and $d_\Gamma(y_i)=d_H(y_i)=d_G(y_i)$ for every $i=1,\ldots,s$.

Moreover, all graphs arising in this Kelmans sequence are connected.
Indeed, the complete bipartite subgraph $H[X,Y_1]=K_{n,n-s-1}$ is
preserved throughout the Kelmans operations, and every vertex of $Y_2$ remains adjacent to some vertex of $X$. Therefore the Kelmans monotonicity of
$\mathcal P$ on connected graphs gives $\mathcal P(H)\le \mathcal P(\Gamma)$.

Since $N_\Gamma(Y_2)\subseteq X_1$, the graph $\Gamma$ has no edge between
$X_2$ and $Y_2$. Hence $\Gamma$ is a subgraph of $S_n^s$ with respect to
the above partition. If $\Gamma\ne S_n^s$, then $\Gamma$ is a proper
spanning subgraph of the connected graph $S_n^s$, and hence the additional
assumption in the theorem gives $\mathcal P(\Gamma)<\mathcal P(S_n^s)$. Thus $\mathcal P(\Gamma)\le \mathcal P(S_n^s)$, with equality only if $\Gamma=S_n^s$. Consequently,
$$
\mathcal P(G)\le \mathcal P(H)\le \mathcal P(\Gamma)
\le \mathcal P(S_n^s)\le M.
$$

We record the equality consequence in this case. If $\mathcal P(G)=M$,
then all inequalities in the last display are equal. Hence $G=H$ and $\Gamma=S_n^s$. In particular, $G[X,Y_1]=K_{n,n-s-1}$. All Kelmans operations from $H$ to $\Gamma$ were performed inside $X$, so the degrees of vertices in $Y$ are preserved. Since $\Gamma=S_n^s$, every vertex of $Y_2$ has degree $s$ in $\Gamma$, and therefore $d_G(y)=s$ for every $y\in Y_2$. Since $G=H$ is connected and $G$ is non-traceable, we have $G\in \mathcal S_n^s$.

\medskip
\noindent\textbf{Case 2.} There exists an integer $t$ with $k\le t\le \left\lfloor (n-2)/2\right\rfloor$ and a set $X_2=\{x_1,\ldots,x_{t+1}\}\subseteq X$ such that $d_G(x_i)\le t$ for every $i=1,\ldots,t+1$.

Choose an arbitrary subset $Y_1\subseteq Y$ with $|Y_1|=t$, and put
$$
Y_2=Y\setminus Y_1,\qquad X_1=X\setminus X_2.
$$
Then
$$
|X_1|=n-t-1,\quad |X_2|=t+1,\quad
|Y_1|=t,\quad |Y_2|=n-t-1.
$$

Let $H'$ be the graph obtained from $G$ by adding all missing edges between $X_1$ and $Y$. Then $H'[X_1,Y]=K_{n-t-1,n-1}$. Since $X_1\ne\emptyset$ and every vertex of $X_2$ remains adjacent to some vertex of $Y$, $H'$ is connected.
The degrees of the vertices in $X_2$ are unchanged, and hence
$d_{H'}(x_i)=d_G(x_i)\le t$ for every $i=1,\ldots,t+1$. If $G\ne H'$, then $G$ is a proper spanning subgraph of the connected graph $H'$, and hence the additional assumption in the theorem gives $\mathcal P(G)<\mathcal P(H')$. Thus $\mathcal P(G)\le \mathcal P(H')$, with equality only if $G=H'$.

Apply Lemma~\ref{le:ch-2.3} to $H'$ with $U=X_2$ and $W=Y_1$. Then there
exists a graph $\Gamma'$ obtained from $H'$ by a finite sequence of
Kelmans operations inside the part $Y$ such that $N_{\Gamma'}(X_2)\subseteq Y_1$ and $d_{\Gamma'}(x_i)=d_{H'}(x_i)=d_G(x_i)$ for every $i=1,\ldots,t+1$.

Moreover, all graphs arising in this Kelmans sequence are connected. Indeed, the complete bipartite subgraph $H'[X_1,Y]=K_{n-t-1,n-1}$ is preserved throughout the Kelmans operations, and every vertex of $X_2$ remains adjacent to some vertex of $Y$. Therefore the Kelmans monotonicity of $\mathcal P$ on connected graphs gives $\mathcal P(H')\le \mathcal P(\Gamma')$.

Since $N_{\Gamma'}(X_2)\subseteq Y_1$, the graph $\Gamma'$ has no edge between $X_2$ and $Y_2$. Hence $\Gamma'$ is a subgraph of $T_n^t$ with respect to the above partition. If $\Gamma'\ne T_n^t$, then $\Gamma'$ is a proper spanning subgraph of the connected graph $T_n^t$, and hence the additional assumption in the theorem gives $\mathcal P(\Gamma')<\mathcal P(T_n^t)$. Thus
$\mathcal P(\Gamma')\le \mathcal P(T_n^t)$, with equality only if $\Gamma'=T_n^t$. Consequently, 
$$
\mathcal P(G)\le \mathcal P(H')\le \mathcal P(\Gamma')
\le \mathcal P(T_n^t)\le M.
$$

We record the equality consequence in this case. If $\mathcal P(G)=M$, then all inequalities in the last display are equal. Hence $G=H'$ and $\Gamma'=T_n^t$. In particular, $G[X_1,Y]=K_{n-t-1,n-1}$. All Kelmans operations from $H'$ to $\Gamma'$ were performed inside $Y$, so the degrees of vertices in $X$ are preserved. Since $\Gamma'=T_n^t$, every vertex of $X_2$ has degree $t$ in $\Gamma'$, and therefore $d_G(x)=t$ for every $x\in X_2$. Since $G=H'$ is connected and $G$ is non-traceable, we have $G\in \mathcal T_n^t$.

In either case, $\mathcal P(G)\le M$. Therefore
$$
\ex_{\mathcal P}(2n-1,P_{2n-1};\delta\ge k)\le M.
$$

Conversely, each graph $S_n^s$ in the indicated range is connected and satisfies $\delta(S_n^s)=\min\{s,n-s-1\}\ge k$. Moreover, $S_n^s$ has no Hamilton path. Indeed, if $S_n^s$ had a Hamilton
path, then adding a new vertex to the smaller part and adding all edges from it to the larger part would produce a Hamilton cycle in the resulting balanced bipartite graph. However, in that balanced graph, deleting $X_1$ leaves the $s$ vertices of $Y_2$ as isolated vertices together with one further component. Thus the remaining graph has $s+1>|X_1|$ components, contradicting the existence of such a Hamilton cycle.

Similarly, each graph $T_n^t$ in the indicated range is connected and satisfies $\delta(T_n^t)=\min\{t,n-t-1\}\ge k$. It also has no Hamilton path. Indeed, if $T_n^t$ had a Hamilton path, then
after adding a new vertex to the smaller part and adding all edges from it to the larger part, the resulting balanced bipartite graph would have a Hamilton cycle. But deleting $Y_1$ together with the new vertex leaves the $t+1$ vertices of $X_2$ as isolated vertices together with one further component. Hence we obtain $t+2>|Y_1|+1$ components after deleting $|Y_1|+1=t+1$ vertices, contradicting the existence of such a Hamilton cycle.

Thus all graphs appearing in the definition of $M$ are admissible for $\ex_{\mathcal P}(2n-1,P_{2n-1};\delta\ge k)$. Hence $\ex_{\mathcal P}(2n-1,P_{2n-1};\delta\ge k)\ge M$. Combining the two inequalities gives $\ex_{\mathcal P}(2n-1,P_{2n-1};\delta\ge k)=M$, which is the desired formula for $\ex_{\mathcal P}(2n-1,P_{2n-1};\delta\ge k)$.

Finally, since the extremal graph $G$ chosen at the beginning was arbitrary, we have $\mathcal P(G)=M$. Hence the equality consequence recorded in Case 1 or Case 2 applies to $G$. Therefore
$$
\EX_{\mathcal P}(2n-1,P_{2n-1};\delta\ge k)
\subseteq
\left(
\bigcup_{k\le s\le \left\lfloor n/2\right\rfloor}
\mathcal S_n^s
\right)
\cup
\left(
\bigcup_{k\le t\le \left\lfloor (n-2)/2\right\rfloor}
\mathcal T_n^t
\right).
$$
This completes the proof.
\end{proof}

For later use, we record the quotient matrices of $S_n^s$ for $1\le s\le n-2$. Let $X=X_1\cup X_2$ and $Y=Y_1\cup Y_2$ be the bipartition of $S_n^s$,
where
\[
|X_1|=s,\quad |X_2|=n-s,\quad |Y_1|=n-s-1,\quad |Y_2|=s.
\]
With respect to the equitable partition $X_1,X_2,Y_1,Y_2$, the adjacency
and signless Laplacian quotient matrices are
\[
B_A(s)=
\begin{pmatrix}
0&0&n-s-1&s\\
0&0&n-s-1&0\\
s&n-s&0&0\\
s&0&0&0
\end{pmatrix},
\qquad
B_Q(s)=
\begin{pmatrix}
n-1&0&n-s-1&s\\
0&n-s-1&n-s-1&0\\
s&n-s&n&0\\
s&0&0&s
\end{pmatrix}.
\]
Set
$$
T_s=n(n-1)-s(n-s),\qquad
D_s=s^2(n-s)(n-s-1),
$$
and define
$$
p_s(z)=z^2-T_sz+D_s.
$$
Also define
$$
g_s(x)=x^3-(3n-2)x^2+
\bigl(2n^2+2ns-3n-2s^2-s+1\bigr)x
-s(2n-1)(n-s-1).
$$
A direct calculation gives
$$
\det(xI-B_A(s))=p_s(x^2)
$$
and
$$
\det(xI-B_Q(s))=xg_s(x).
$$
Thus, if $\mu_s=\rho(S_n^s)^2$, then $\mu_s$ is the larger root of $p_s(z)=0$. 
Moreover, if $\theta_s=q(S_n^s)$, then $\theta_s$ is the largest root of $g_s(x)=0$.

\begin{lem}\label{le:ch-4.2}
Let $n\ge 2k+1$ and $k\ge1$. Then
\[
\max_{k\le s\le n-k-1}\rho(S_n^s)=\rho(S_n^k),
\]
and equality is attained only when $s=k$. In particular,
\[
\max\left(
\{\rho(S_n^s):k\le s\le \lfloor n/2\rfloor\}
\cup
\{\rho(T_n^t):k\le t\le \lfloor(n-2)/2\rfloor\}
\right)
=
\rho(S_n^k),
\]
and equality in this maximum is attained only by $S_n^k$.
\end{lem}

\begin{proof}
Use the notation introduced above. We show that $\mu_s<\mu_k$ for every
$s>k$ with $s\le n-k-1$. Since $S_n^k$ properly contains
$K_{n,n-k-1}\cup kK_1$, we have $\mu_k>n(n-k-1)$.

Let $f(u)=u^2(n-u)(n-u-1)$. For $k\le u\le n-k-1$, the concave functions
$u(n-u)$ and $u(n-u-1)$ attain their minima at the endpoints of this
interval. Hence $u(n-u)\ge k(n-k)$ and
$u(n-u-1)\ge k(n-k-1)$, so $f(u)\ge f(k)$.

Now fix $s>k$. Since $p_k(\mu_k)=0$, we have
$$
p_s(\mu_k)
=
p_s(\mu_k)-p_k(\mu_k)
=
(T_k-T_s)\mu_k+D_s-D_k.
$$
By the definitions of $T_s$ and $D_s$, we have
$$
T_k-T_s=s(n-s)-k(n-k)=(s-k)(n-s-k)
$$
and
$$
D_s-D_k=f(s)-f(k).
$$
Hence
$$
p_s(\mu_k)=(s-k)(n-s-k)\mu_k+f(s)-f(k).
$$
Since $s>k$, $s\le n-k-1$, and $f(s)\ge f(k)$, we get $p_s(\mu_k)>0$.
Moreover,
\[
p_s'(\mu_k)=2\mu_k-T_s
>
2n(n-k-1)-T_s
=
n(n-2k-1)+s(n-s)>0.
\]
Thus $\mu_k$ lies to the right of the larger root of $p_s$, and hence
$\mu_s<\mu_k$. Therefore $\rho(S_n^s)<\rho(S_n^k)$ for every
$s>k$ with $s\le n-k-1$, proving the first assertion and its equality case.

Finally, $T_n^t\cong S_n^{\,n-t-1}$. If
$k\le t\le\lfloor(n-2)/2\rfloor$ and $s'=n-t-1$, then
$k<s'\le n-k-1$. Hence
\[
\rho(T_n^t)=\rho(S_n^{s'})<\rho(S_n^k).
\]
The asserted maximum over the union and the uniqueness of the equality case follow. This proves the lemma.
\end{proof}

\begin{lem}\label{le:ch-4.3}
Let $n\ge 2k+1$ and $k\ge1$. Then
\[
\max_{k\le s\le n-k-1}q(S_n^s)=q(S_n^k),
\]
and equality is attained only when $s=k$. In particular,
\[
\max\left(
\{q(S_n^s):k\le s\le \lfloor n/2\rfloor\}
\cup
\{q(T_n^t):k\le t\le \lfloor(n-2)/2\rfloor\}
\right)
=
q(S_n^k),
\]
and equality in this maximum is attained only by $S_n^k$.
\end{lem}

\begin{proof}
Use the notation introduced above. Since $S_n^k$ properly contains
$K_{n,n-k-1}\cup kK_1$, we have
\[
\theta_k>q(K_{n,n-k-1})=2n-k-1.
\]
Let $s>k$ with $s\le n-k-1$, and put $r=n-s-k-1\ge0$. Since
$g_k(\theta_k)=0$, we get
$$
g_s(\theta_k)
=
(s-k)\bigl((2r+1)\theta_k-(2n-1)r\bigr).
$$
As $\theta_k>2n-k-1$, it follows that
$$
g_s(\theta_k)
>
(s-k)\bigl((2n-2k-1)r+2n-k-1\bigr)>0.
$$

It remains to show that $g_s$ has no root larger than $\theta_k$. We have
$g_s''(x)=6x-6n+4>0$ for $x\ge 2n-k-1$. Writing $h=s-k$, a direct
calculation gives
\[
g_s'(2n-k-1)
=
2h^2+4hk+6hr+4h+k^2+4kr+3k+2r^2+3r+1>0.
\]
Hence $g_s$ is strictly increasing on $[2n-k-1,\infty)$. Since
$\theta_k>2n-k-1$ and $g_s(\theta_k)>0$, the largest root of $g_s$ is
smaller than $\theta_k$. Therefore
$q(S_n^s)<q(S_n^k)$ for every $s>k$ with $s\le n-k-1$, proving the first
assertion and its equality case.

Finally, $T_n^t\cong S_n^{\,n-t-1}$. If
$k\le t\le\lfloor(n-2)/2\rfloor$ and $s'=n-t-1$, then
$k<s'\le n-k-1$. Hence
\[
q(T_n^t)=q(S_n^{s'})<q(S_n^k).
\]
The asserted maximum over the union and the uniqueness of the equality case follow. This proves the lemma.
\end{proof}

\begin{lem}\label{le:ch-4.4}
Let $1\le s\le \left\lfloor n/2\right\rfloor$ and let
$G\in\mathcal S_n^s$. Then
\[
\rho(G)\le \rho(S_{n}^{s})
\qquad\text{and}\qquad
q(G)\le q(S_{n}^{s}),
\]
with equality only when $G\cong S_{n}^{s}$.

Similarly, let $1\le t\le \left\lfloor (n-2)/2\right\rfloor$ and let
$G\in\mathcal T_n^t$. Then 
\[
\rho(G)\le \rho(T_{n}^{t})
\qquad\text{and}\qquad
q(G)\le q(T_{n}^{t}),
\]
with equality only when $G\cong T_{n}^{t}$.
\end{lem}

\begin{proof}
We first prove the assertion for $\mathcal S_n^s$. Let
$G=G[X,Y]\in\mathcal S_n^s$. By the definition of $\mathcal S_n^s$, there
exists a set $Y_2\subseteq Y$ with $|Y_2|=s$ such that, putting
$Y_1=Y\setminus Y_2$, we have $G[X,Y_1]=K_{n,n-s-1}$ and $d_G(y)=s$ 
for each $y\in Y_2$.

First consider the adjacency spectral radius. Let $\mathbf x$ be the
positive Perron vector of $G$. Choose a set $X_1\subseteq X$ with
$|X_1|=s$ such that the vertices in $X_1$ have the largest
$\mathbf x$-coordinates among all vertices of $X$, and put
$X_2=X\setminus X_1$.

Suppose that there exist $y\in Y_2$, $u\in X_2$, and $v\in X_1$ such that
$uy\in E(G)$ and $vy\notin E(G)$. Let $G'$ be obtained from the current
graph by deleting the edge $uy$ and adding the edge $vy$. Since
$x_v\ge x_u$, we have
$$
\mathbf x^T A(G')\mathbf x-\mathbf x^T A(G)\mathbf x
=
2x_y(x_v-x_u)\ge0.
$$
Thus this switching does not decrease the adjacency Rayleigh quotient with
respect to $\mathbf x$.

We repeat this operation as long as there is an edge between $Y_2$ and
$X_2$. The process terminates after finitely many steps, because each
switch decreases the number of edges between $Y_2$ and $X_2$. At the end,
every vertex of $Y_2$ has all its neighbors in $X_1$. Since every vertex of
$Y_2$ has degree $s$ and $|X_1|=s$, every vertex of $Y_2$ is adjacent to
all vertices of $X_1$. Together with $G[X,Y_1]=K_{n,n-s-1}$, the resulting
graph is $S_n^s$ with respect to the partition
$X=X_1\cup X_2$ and $Y=Y_1\cup Y_2$. Therefore $\rho(G)\le \rho(S_n^s)$.

Now suppose that equality holds. Then equality must hold at every step in
the above Rayleigh quotient comparison. In particular, if a nontrivial
switch from $u\in X_2$ to $v\in X_1$ was made, then $x_v=x_u$. 
Moreover, equality in the final Rayleigh quotient implies that the original vector $\mathbf x$ is also a Perron vector of the final graph $S_n^s$.

In $S_n^s$, the Perron coordinates are constant on each of the four classes
$X_1,X_2,Y_1,Y_2$. Let these coordinates be $x_1,x_2,y_1,y_2$, respectively.
Then the eigenvalue equations for vertices in $X_1$ and $X_2$ give
$$
\rho x_1=(n-s-1)y_1+s y_2
$$
and
$$
\rho x_2=(n-s-1)y_1.
$$
Since $y_2>0$, we get $x_1>x_2$. This contradicts the equality condition
$x_v=x_u$ for any nontrivial switch from $X_2$ to $X_1$. Hence no
nontrivial switch can occur under equality. Therefore $G\cong S_n^s$.

Next consider the signless Laplacian spectral radius. Let $\mathbf z$ be
the positive Perron vector of $Q(G)$. Choose again a set $X_1\subseteq X$
with $|X_1|=s$, now so that the vertices in $X_1$ have the largest
$\mathbf z$-coordinates among all vertices of $X$, and put
$X_2=X\setminus X_1$.

Perform the same switching process. If an edge $uy$ with $u\in X_2$ and
$y\in Y_2$ is replaced by an edge $vy$ with $v\in X_1$, then the signless
Laplacian Rayleigh quotient changes by
$$
(z_v+z_y)^2-(z_u+z_y)^2.
$$
Since $z_v\ge z_u$, this quantity is nonnegative. Hence the switching
process does not decrease the signless Laplacian Rayleigh quotient. As
above, the final graph is $S_n^s$, and therefore $q(G)\le q(S_n^s)$.

Suppose that equality holds for $q$. Then every nontrivial switch must have
zero Rayleigh quotient increment, and hence it must satisfy $z_v=z_u$.
Moreover, the original vector $\mathbf z$ must also be a Perron vector of
$Q(S_n^s)$.

Let $z_1,z_2,w_1,w_2$ be the signless Laplacian Perron coordinates of
$S_n^s$ on $X_1,X_2,Y_1,Y_2$, respectively. The eigenvalue equations for
vertices in $X_1$ and $X_2$ give
$$
(q-n+1)z_1=(n-s-1)w_1+s w_2
$$
and
$$
(q-n+s+1)z_2=(n-s-1)w_1.
$$
Since $w_2>0$, and since the first equation implies $q-n+1>0$, we have
$$
z_1
=
\frac{(n-s-1)w_1+s w_2}{q-n+1}
>
\frac{(n-s-1)w_1}{q-n+s+1}
=
z_2.
$$
Thus the Perron coordinate on $X_1$ is strictly larger than that on $X_2$.
This contradicts the equality condition $z_v=z_u$ for any nontrivial switch
from $X_2$ to $X_1$. Hence equality for $q$ also forces
$G\cong S_n^s$.

The proof for $\mathcal T_n^t$ is analogous, but we spell out the switching
direction. Let $G=G[X,Y]\in\mathcal T_n^t$. By the definition of
$\mathcal T_n^t$, there exists a set $X_2\subseteq X$ with $|X_2|=t+1$
such that, putting $X_1=X\setminus X_2$, we have $G[X_1,Y]=K_{n-t-1,n-1}$ and $d_G(x)=t$ for every $x\in X_2$.

For the adjacency spectral radius, choose a set $Y_1\subseteq Y$ with $|Y_1|=t$ consisting of the vertices of $Y$ with the largest coordinates in the positive Perron vector of $G$, and put $Y_2=Y\setminus Y_1$. Whenever some $x\in X_2$ is adjacent to $u\in Y_2$ and nonadjacent to $v\in Y_1$, replace the edge $xu$ by $xv$. This does not decrease the adjacency Rayleigh quotient. Repeating the switching process gives the canonical graph $T_n^t$. Hence
$\rho(G)\le \rho(T_n^t)$. If equality holds and a nontrivial switch was made, then the switched vertices in $Y$ must have equal Perron coordinates in the final graph $T_n^t$. But in $T_n^t$, if $x_1,x_2,y_1,y_2$ denote the adjacency Perron coordinates on $X_1,X_2,Y_1,Y_2$, respectively, then
$$
\rho y_1=(n-t-1)x_1+(t+1)x_2
$$
and
$$
\rho y_2=(n-t-1)x_1.
$$
Since $x_2>0$, we have $y_1>y_2$. Thus equality is impossible after a
nontrivial switch, and hence equality holds only when $G\cong T_n^t$.

For the signless Laplacian spectral radius, choose $Y_1$ with respect to
the positive Perron vector of $Q(G)$ and perform the same switching. Each
switch from $u\in Y_2$ to $v\in Y_1$ changes the signless Laplacian Rayleigh
quotient by a nonnegative amount
$$
(z_x+z_v)^2-(z_x+z_u)^2.
$$
Thus $q(G)\le q(T_n^t)$.
For the equality case, let $z_1,z_2,w_1,w_2$ be the signless Laplacian
Perron coordinates of $T_n^t$ on $X_1,X_2,Y_1,Y_2$, respectively. The
eigenvalue equations for vertices in $Y_1$ and $Y_2$ give
$$
(q-n)w_1=(n-t-1)z_1+(t+1)z_2
$$
and
$$
(q-n+t+1)w_2=(n-t-1)z_1.
$$
Since $z_2>0$, and since the first equation implies $q-n>0$, we have
$$
w_1=\frac{(n-t-1)z_1+(t+1)z_2}{q-n}>\frac{(n-t-1)z_1}{q-n+t+1}=w_2.
$$
Thus a nontrivial switch is incompatible with equality. Therefore equality for $q$ holds only when $G\cong T_n^t$.

This completes the proof.
\end{proof}

\begin{proof}[Proof of Theorem~\ref{th:ch-1.9}]
Let $\lambda\in\{\rho,q\}$. Since $n\ge 2k+1$, we have
$1\le k\le \lfloor (n-1)/2\rfloor$. Both $\rho$ and $q$ are feasible graph
parameters. Moreover, for $\mathcal P=\rho$ or $\mathcal P=q$, the condition $\mathcal P(H)<\mathcal P(G)$ whenever $H$ is a proper spanning subgraph of a connected graph $G$ follows from the Perron--Frobenius Theorem. Since $S_{n}^{k}$ is non-traceable and
$\delta(S_{n}^{k})=k$, we have
\[
\spex_{\lambda}^{\mathrm{NB}}(2n-1,P_{2n-1};\delta\ge k)
\ge \lambda(S_{n}^{k}).
\]

Let $G$ be a non-traceable nearly balanced bipartite graph with part
sizes $n$ and $n-1$ and minimum degree at least $k$. Applying Theorem~\ref{th:ch-1.8} with $\mathcal P=\lambda$, we get 
\[
\lambda(G)\le
\max\left(
\{\lambda(S_{n}^{s}): k\le s\le \lfloor n/2\rfloor\}
\cup
\{\lambda(T_{n}^{t}): k\le t\le \lfloor (n-2)/2\rfloor\}
\right).
\]
By Lemma~\ref{le:ch-4.2} when $\lambda=\rho$, and by
Lemma~\ref{le:ch-4.3} when $\lambda=q$, the right-hand side is
$\lambda(S_{n}^{k})$. Hence
\[
\spex_{\lambda}^{\mathrm{NB}}(2n-1,P_{2n-1};\delta\ge k)
=
\lambda(S_{n}^{k}).
\]

The equality case follows from the equality part of
Theorem~\ref{th:ch-1.8}, Lemma~\ref{le:ch-4.4}, and the
uniqueness in Lemma~\ref{le:ch-4.2} or
Lemma~\ref{le:ch-4.3}. Therefore equality holds only when
$G\cong S_{n}^{k}$, and so
\[
\SPEX_{\lambda}^{\mathrm{NB}}(2n-1,P_{2n-1};\delta\ge k)
=
\{S_{n}^{k}\}.
\]
Taking $\lambda=\rho$ and $\lambda=q$ gives the two assertions. This completes the proof.
\end{proof}

\section*{Declaration of competing interest}

There is no competing interest.

\section*{Acknowledgements}

This research was supported in part by the National Key Research and
Development Program of China (No. 2023YFA1009604), the National Natural Science Foundation of China
(No. 12371350), and the Fundamental
Research Funds for the Central Universities, Nankai University
(No. 63243151). The author is deeply grateful to Professor Bo Ning of
Nankai University for his invaluable guidance and many helpful discussions.
The problem studied in this paper was proposed by Professor Ning during a
seminar, and his insightful suggestions played an important role in shaping
this work.

\section*{Data availability}

No data was used for the research described in the article.

\end{document}